\documentclass[12pt,reqno]{amsart}
\usepackage{amsmath,amsthm,amssymb,amsfonts,amscd}
\usepackage{mathrsfs}
\usepackage{bbm}
\usepackage{bbding}
\usepackage[backref=page]{hyperref}
\usepackage{geometry}
\usepackage{color}
\usepackage{xcolor}

\usepackage{picture,epic}
\usepackage{tikz}

\numberwithin{equation}{section}

\theoremstyle{plain}
\newtheorem{theorem}{Theorem}
\newtheorem{lemma}[theorem]{Lemma}

\newtheorem{proposition}[theorem]{Proposition}

\theoremstyle{definition}

\theoremstyle{remark}
\newtheorem{remark}[theorem]{Remark}

\renewcommand{\Re}{\operatorname{Re}}

\newcommand{\supp}{\operatorname{supp}}

\newcommand{\Sym}{\operatorname{Sym}}

\newcommand{\GL}{\operatorname{GL}}
\newcommand{\SL}{\operatorname{SL}}

\newcommand{\dd}{\mathrm{d}}

\makeatletter
\def\@tocline#1#2#3#4#5#6#7{\relax
  \ifnum #1>\c@tocdepth 
  \else
    \par \addpenalty\@secpenalty\addvspace{#2}%
    \begingroup \hyphenpenalty\@M
    \@ifempty{#4}{%
      \@tempdima\csname r@tocindent\number#1\endcsname\relax
    }{%
      \@tempdima#4\relax
    }%
    \parindent\z@ \leftskip#3\relax \advance\leftskip\@tempdima\relax
    \rightskip\@pnumwidth plus4em \parfillskip-\@pnumwidth
    #5\leavevmode\hskip-\@tempdima
      \ifcase #1
       \or\or \hskip 1em \or \hskip 2em \else \hskip 3em \fi%
      #6\nobreak\relax
    \hfill\hbox to\@pnumwidth{\@tocpagenum{#7}}\par
    \nobreak
    \endgroup
  \fi}
\makeatother

\begin{document}

\title
{On the Rankin--Selberg problem, II}
\author{Bingrong Huang}
\address{Data Science Institute and School of Mathematics \\ Shandong University \\ Jinan \\ Shandong 250100 \\China}
\email{brhuang@sdu.edu.cn}

\dedicatory{Dedicated to Xiumin Ren on the occasion of her 60th birthday}

\date{\today}

\begin{abstract}
  In this paper, we improve our bounds on the Rankin--Selberg problem. That is, we obtain smaller error term of the second moment of Fourier coefficients of a $\GL(2)$ cusp form (both holomorphic and Maass).
\end{abstract}



\thanks{This work was supported by  the National Key R\&D Program of China (No. 2021YFA1000700) and NSFC (Nos. 12001314 and 12031008).}

\maketitle

\section{Introduction} \label{sec:Intr}

This is a continuation of our paper \cite{huang2021rankin}, where we break the barrier $3/5$ on the Rankin--Selberg problem \cite{Rankin1939,Selberg1940}. In this paper, we will refine the approach and improve the main result there.

Let $\phi$ be a  holomorphic Hecke cusp form or
Hecke--Maass cusp form for $\SL(2,\mathbb{Z})$.
Let $\lambda_\phi(n)$ be its $n$-th Hecke eigenvalue. We define
\[
  \mathcal S_2(X,\phi) := \sum_{n\leq X} \lambda_\phi(n)^2
  \quad \textrm{and} \quad
  \Delta_2(X,\phi) := \mathcal S_2(X,\phi) - c_\phi X,
\]
where $c_\phi = L(1,\Sym^2 \phi)/\zeta(2)$ and $L(s,\Sym^2 \phi)$ is the symmetric square $L$-function of $\phi$.
In 1939/1940, Rankin \cite{Rankin1939} and Selberg  \cite{Selberg1940} invented the powerful Rankin--Selberg method, and successfully showed that
\[
  \Delta_2(X,\phi) \ll_\phi X^{3/5}.
\]
This bound remained the best since it was proved for more than 80 years.
The \emph{Rankin--Selberg problem} is to improve the exponent $3/5$.
In \cite{huang2021rankin}, we proved
  \begin{equation}\label{eqn:Delta<<}
    \Delta_2(X,\phi) \ll_{\phi,\delta} X^{3/5-\delta},
  \end{equation}
  for any $\delta < 1/560 =0.001785...$.
In a recent preprint \cite{pal}, Pal proved that $\delta<6/1085=0.005529...$ is admissible in \eqref{eqn:Delta<<} by using nontrivial bounds for the integral second moment of degree three $L$-function.
In this paper, we prove the following theorem.
\begin{theorem}\label{thm:RS}
  In \eqref{eqn:Delta<<}, any $\delta<3/305=0.009836...$ is admissible.
\end{theorem}

To prove Theorem \ref{thm:RS}, comparing to \cite{huang2021rankin}, we avoid the use of the theory of exponential pairs.
  Instead, we should make use of the bilinear structure in the dual sum of $\mathcal S_2(X,\phi)$.
  The result above is as strong as what we can get in \cite{huang2021rankin} with the exponential pair conjecture. See \cite[Remark 2]{huang2021rankin} with $(k,\ell)=(0,1/2)$.
  Our new treatment  is motivated by a recent work \cite{KLM} of Kowalski--Lin--Michel, where they considered the Rankin--Selberg coefficients in large arithmetic progressions.

\begin{remark}
  We did not try to obtain the best exponent. For example, one may apply the recent method of Dasaratharaman--Munshi \cite{DM} to improve bounds for the exponential sums involving the coefficients of a self dual $\GL(3)$ automorphic form, which may improve our Theorem \ref{thm:RS} a little bit. But this may not work for generic $\GL(3)$ forms as in Theorem \ref{thm:1+3} below.
\end{remark}

Let
\begin{equation*}
 L(s,\phi\times \phi) = \zeta(2s) \sum_{n=1}^{\infty} \frac{\lambda_\phi(n)^2}{n^s}, \quad \Re(s)>1
\end{equation*}
be the Rankin--Selberg $L$-function of $\phi\times \phi$.
Note that $\phi\times \phi = 1 \boxplus \Sym^2 \phi$, where $\Sym^2 \phi$ is the symmetric square lift of $\phi$. That is $L(s,\phi\times \phi) = \zeta(s)L(s,\Sym^2 \phi)$.
In \cite{GelbartJacquet}, Gelbart--Jacquet proved that $\Sym^2 \phi$ is an automorphic cuspidal representation for $\GL(3)$.
As in \cite{huang2021rankin}, to prove Theorem \ref{thm:RS}, we first consider
\[
  \mathcal A(X,\phi\times \phi) :=\sum_{n\leq X} \lambda_{\phi\times \phi}(n)
  = \sum_{n\leq X} \lambda_{1 \boxplus \Sym^2 \phi}(n).
\]
More generally, we consider $\mathcal A(X,1\boxplus \Phi) = \sum_{n\leq X} \lambda_{1 \boxplus \Phi}(n)$, where  $\Phi$ is a Hecke--Maass cusp form for $\SL(3,\mathbb Z)$ with $A_\Phi(m,n)$  the normalized Fourier coefficients of $\Phi$. The \emph{generalized Ramanujan conjecture} (GRC) for $\Phi$ asserts that $A_\Phi(1,n) \ll n^{o(1)}$.
We prove the following result.

\begin{theorem}\label{thm:1+3}
  Assuming GRC for $\Phi$. Then we have
  \begin{equation*}
    \mathcal A(X,1\boxplus\Phi)  = L(1,\Phi) \, X + O_{\Phi,\delta}( X^{3/5-\delta} ),
  \end{equation*}
  for any $\delta < 3/305=0.009836...$.

  If $\Phi=\Sym^2 \phi$, then we don't need to assume GRC for $\Phi$.
\end{theorem}

Our  method to prove Theorem \ref{thm:1+3} can be used to improve other asymptotic formulas in this shape. For example, one can improve Huang--Sun--Zhang \cite[Cor. 1.6]{HSZ} and Zhang \cite{Zhang}, which may appear in Zhang's thesis.

\medskip

The proof of Theorem \ref{thm:RS} by using Theorem \ref{thm:1+3} is the same as \cite[\S4]{huang2021rankin}.

\medskip
The plan of this paper is as follows.
In \S\ref{sec:preliminaries}, we give a review of Maass forms and $L$-functions, the analytic twisted sums of $\GL(3)$ coefficients, and the stationary phase method.
In \S\ref{sec:proof}, we prove Theorem \ref{thm:1+3}.

\medskip
\textbf{Notation.}
Throughout the paper, $\varepsilon$ is an arbitrarily small positive number;
all of them may be different at each occurrence.
The weight functions $U, \ V,\ W$ may also change at each occurrence.
As usual, $e(x)=e^{2\pi i x}$.

\section{Preliminaries}\label{sec:preliminaries}

\subsection{Maass forms and $L$-functions} \label{subsec:L-functions}

Let $\Phi$ be a Hecke--Maass form of type $(\nu_1,\nu_2)$ for $\SL(3,\mathbb{Z})$ with the normalized Fourier coefficients $A(m,n)$ such that $A(1,1)=1$.
The Langlands parameters are defined as $\alpha_1 = -\nu_1 - 2\nu_2 +1$,
$\alpha_2 = -\nu_1 +\nu_2$, and $\alpha_3 = 2\nu_1 +\nu_2 -1$.
It is well known that by standard properties of the Rankin--Selberg $L$-function we have the Ramanujan conjecture on average
\begin{equation}\label{eqn:RS3}
  \mathop{\sum_{m\geq1}\sum_{n\geq1}}_{m^2 n \leq N} |A(m,n)|^2
  \ll_\Phi N^{1+\varepsilon}.
\end{equation}
The $L$-function associated with $\Phi$ is given by
$L(s,\Phi) = \sum_{n=1}^{\infty} A(1,n) n^{-s}$ in the domain $\Re(s)>1$. It extends to an entire function and satisfies the following functional equation
\[
  \gamma(s,\Phi) L(s,\Phi) = \gamma(1-s,\tilde\Phi) L(1-s,\tilde{\Phi}), \quad \gamma(s,\Phi) = \prod_{j=1}^{3} \pi^{-s/2} \Gamma\left(\frac{s-\alpha_j}{2}\right).
\]
Here $\tilde\Phi$ is the dual form having Langlands parameters $(-\alpha_3,-\alpha_2,-\alpha_1)$ and the Fourier coefficients $\overline{A(m,n)}$. See more information in Goldfeld \cite{goldfeld2006automorphic}.

\subsection{Analytic twisted sums of $\GL(3)$ coefficients}\label{subsec:AT3}

In order to prove Theorem  \ref{thm:1+3}, we will use power saving upper bounds for the analytic twisted sum of $\GL(3)$ Fourier coefficients.
Define
\[
  \mathscr{S}(N) := \sum_{n\geq1} A_\Phi(1,n) e\left(T \varphi\left(\frac{n}{N}\right)\right) V\left(\frac{n}{N}\right),
\]
where $T\geq1$ is a large parameter, $\varphi$ is some fixed real-valued smooth function, and $V\in C_c^\infty (\mathbb{R})$ with $\supp V\subset [1/2,1]$, total variation $\mathrm{Var}(V)\ll 1$ and satisfying that $V^{(k)} \ll P^k$ for all $k\geq0$ with $P\ll T^{\eta}$ for some small $\eta\in [0,1/10]$.

\begin{lemma}\label{lemma:3}
  Assume $\varphi(u)=u^\beta$ with $\beta\in (0,1)$.
  Then we have
  \[
    \mathscr S(N) \ll_{\Phi,\varepsilon} T^{3/10} N^{3/4+\varepsilon},
  \]
  if $T^{6/5} \leq N \leq T^{8/5-\varepsilon}$.
\end{lemma}

\begin{proof}
  This is \cite[Theorem 9]{huang2021rankin}.
\end{proof}

\subsection{Oscillatory integrals}

Let $\mathcal{F}$ be an index set and $X=X_T:\mathcal{F}\rightarrow \mathbb{R}_{\geq1}$ be a function of $T\in\mathcal{F}$. A family of $\{w_T\}_{T\in\mathcal{F}}$ of smooth functions supported on a product of dyadic intervals in $\mathbb{R}_{>0}$ is called $X$-inert if for each $j\in \mathbb{Z}_{\geq0}$ we have
\[
  \sup_{T\in\mathcal{F}} \sup_{x \in \mathbb{R}_{>0}}
  X_T^{-j} \left| x^{j} w_T^{(j)} (x) \right|
   \ll_{j} 1.
\]
We will use the following stationary phase lemma. In fact, we only need upper bounds.

\begin{lemma}\label{lemma:stationary_phase}
  Suppose $w=w_T$ is a family of $X$-inert in $\xi$ with compact support on $[Z,2Z]$,
  so that $w^{(j)}(\xi) \ll (Z/X)^{-j}$.
  Suppose that on the support of $w$, $h=h_T$ is smooth and satisfies that
  $ h^{(j)}(\xi) \ll \frac{Y}{Z^j}$, for all $j\geq0.$
  Let
  \[
    I = \int_{\mathbb{R}} w(\xi) e^{i h(\xi)}  \dd \xi.
  \]
  \begin{enumerate}
  \item [(i)] If $h'(\xi) \gg \frac{Y}{Z}$ for all $\xi \in \supp w$. Suppose $Y/X \geq1$. Then
      $I \ll_A Z (Y/X)^{-A}$ for $A$ arbitrarily large.
  \item [(ii)] If $h''(\xi) \gg \frac{Y}{Z^2}$
  for all $\xi \in \supp w$, and there exists $\xi_0 \in\mathbb{R}$ such that $h'(\xi_0)=0$.
  Suppose that $Y/X^2 \geq 1$. Then we have
  \[
    I
    = \frac{e^{i h(\xi_0)}}{\sqrt{h''(\xi_0)}}  W_T(\xi_0) + O_A(Z(Y/X^2)^{-A}),
    \quad
    \textrm{for any $A>0$,}
  \]
  for some $X$-inert family of functions $W_T$ (depending on $A$) supported on $\xi_0\asymp Z$.
  \end{enumerate}
\end{lemma}


\begin{proof}
  See \cite[\S 8]{BlomerKhanYoung}.
\end{proof}

\section{Proof of Theorem \ref{thm:1+3} }\label{sec:proof}

Let $0<Y\leq X/5$. Let $W_1$ be a smooth function with support $\supp W_1 \in [1/2-Y/X,1+Y/X]$ such that $W_1(u) = 1$ if $u\in[1/2,1]$, $W_1(u)\in[0,1]$ if $u\in[1/2-Y/X,1/2]\cup [1,1+Y/X]$.
Similarly, let $W_2$ be a smooth function with support $\supp W_2 \in [1/2,1]$ such that $W_2(u) = 1$ if $u\in[1/2+Y/X,1-Y/X]$, $W_2(u)\in[0,1]$ if $u\in[1/2,1/2+Y/X]\cup [1-Y/X,1]$. Assume  $W_j^{(k)}(u) \ll (X/Y)^{k}$ for any integer $k\geq0$ and $j\in\{1,2\}$.
By the same argument as in \cite[\S3.1]{huang2021rankin}, to prove the first claim in Theorem \ref{thm:1+3}, it suffices to show for $W=W_1$ and $Y=X^{3/5-\delta}$ with  $\delta=3/305$,  we have
\begin{equation}\label{eqn:smoothed}
  \sum_{n\geq1} \lambda_{1\boxplus\Phi}(n) W\left(\frac{n}{X}\right) = L(1,\Phi) \tilde{W}(1) X  +  O(X^{3/5-\delta+o(1)}),
\end{equation}
where $\tilde{W}(s)$ 
is the Mellin transform of $W$.
Let $\phi$ be a Hecke--Maass cusp form for $\SL(2,\mathbb Z)$. Consider the case $\Phi=\Sym^2 \phi$.
To prove the second claim in Theorem \ref{thm:1+3}, it suffices to prove \eqref{eqn:smoothed} for any $W\in\{W_1,W_2\}$.
Note that $\tilde{W}(1)=1/2+O(Y/X)$.

In \cite[Eq. (3.3)]{huang2021rankin}, we prove
\begin{multline}\label{eqn:sum2integral}
  \sum_{n\geq1} \lambda_{1\boxplus\Phi}(n) W\left(\frac{n}{X}\right)
  = L(1,\Phi) X/2 + O(Y) \\
   + O\bigg( \sup_{u\in [1/3,2]} \sup_{T\ll X^{1+\varepsilon}/Y } \frac{X^{1/2}}{T} \left| \int_{\mathbb{R}} V\left(\frac{t}{T}\right) L(1/2+it,1\boxplus\Phi) (uX)^{it} \dd t \right| \bigg),
\end{multline}
for some fixed $V$ with compact support.
Here $L(s,1\boxplus\Phi)=\zeta(s)L(s,\Phi)$ and $\zeta(s)$ is the Riemann zeta function.
Hence it suffices to consider
\[
  \mathcal{I} := \mathcal{I}_\Phi(T,X)=\int_{\mathbb{R}} V\left(\frac{t}{T}\right) L(1/2+it,1\boxplus\Phi) X^{it} \dd t.
\]
We only consider the case $T\geq1$, since the case $T\leq -1$ can be done similarly and the case $-1\leq T\leq 1$ can be treated trivially. We have the following proposition.
\begin{proposition}\label{prop:I}
  Let $\Phi$ be a Hecke--Maass cusp form for $\SL_3(\mathbb{Z})$.
  We have
  \begin{itemize}
    \item [(i)] For any $X>0$ and $T\geq1$, we have $\mathcal{I} \ll T^{5/4+\varepsilon}$.
    \item [(ii)] Assume  $X$ is sufficiently large.
  If $X^{5/14+\varepsilon} \leq T \leq X^{5/12-\varepsilon}$, then we have
  \[
    \mathcal{I} \ll T^{61/30+\varepsilon} X^{-1/3}.
  \]
  \end{itemize}
\end{proposition}

We will prove this proposition in the next subsection.
Now we can complete the proof of Theorem \ref{thm:1+3}.

\begin{proof}[Proof of Theorem  \ref{thm:1+3} by assuming Proposition \ref{prop:I}]
  Assume $\delta\in(0,1/100)$. Then by Proposition \ref{prop:I} (ii), for $X^{2/5-4\delta} \leq T \leq 
   X^{2/5+\delta+\varepsilon}$, the contribution to \eqref{eqn:sum2integral} is bounded by
  \[
    O(X^{1/2}T^{-1}  T^{61/30+\varepsilon} X^{-1/3})
    = O(X^{1/6} T^{31/30+\varepsilon}) = O(X^{87/150+31\delta/30+\varepsilon}).
  \]
  Note that the other error term is $O(X^{3/5-\delta+\varepsilon})$, so the best choice is $\delta=3/305$.
  By Proposition \ref{prop:I} (i), for $T \leq X^{2/5-4\delta}$, the contribution to \eqref{eqn:sum2integral} is bounded by
  \[
    O(X^{1/2}T^{-1} T^{5/4+\varepsilon})
    = O(X^{1/2}T^{1/4+\varepsilon})  = O(X^{3/5-\delta+\varepsilon}).
  \]
  This proves Theorem  \ref{thm:1+3}.
\end{proof}

\subsection{The integral first moment}

Proposition \ref{prop:I} (i) is standard. See e.g. \cite[Prop. 19 (i)]{huang2021rankin}.

Proposition \ref{prop:I} (ii) is an improvement of  \cite[Prop. 19 (iii)]{huang2021rankin}.
The first several steps of the proof are the same.

For $T\geq X^{1/4+\varepsilon}$, by
\cite[Eq. (3.9)]{huang2021rankin}, we have
\begin{equation}\label{eqn:I<<sum}
  \mathcal{I} \ll  T^\varepsilon \sup_{w\in[\varepsilon-iT^\varepsilon,\varepsilon+iT^\varepsilon]}
  \sup_{N\asymp T^{4}/X} \sup_{|v|\leq T^\varepsilon}
   \frac{T^{1/2}}{ N^{1/2}}
  \Big| \sum_{n\geq1}  \lambda_{1\boxplus\Phi}(n) V\left(\frac{n}{N}\right) e\left(-4(nX)^{1/4}\right) \Big|
  +  T^{1+\varepsilon} .
\end{equation}
Here $V$  is certain smooth function (depending on $v$) such that $\supp V\subset (1/20,20)$ and $V^{(k)} \ll T^{k\varepsilon}$ for $k\geq0$.
To prove Proposition \ref{prop:I} (ii), we need to find a power saving in the dual sum
\[
  \mathcal B(N) := \sum_{n\geq1}  \lambda_{1\boxplus\Phi}(n) V\left(\frac{n}{N}\right) e\left(-4(nX)^{1/4}\right)
\]
if $X^{5/14+\varepsilon} \leq T \leq X^{5/12-\varepsilon}$.
Now we should use the fact $\lambda_{1\boxplus\Phi}(n) = \sum_{\ell m=n}A(1,m)$. Applying smooth dyadic partitions of the $\ell$-sum and  the $m$-sum, we get
\[
  \mathcal B(N) = \sum_{\substack{L\ll N \\ L \; \textrm{dyadic}}} \sum_{\substack{M\ll N \\ M \; \textrm{dyadic}}} \sum_{ \ell\geq 1} \sum_{m\geq1}  A(1,m)e\left(- 4(\ell mX)^{1/4}\right) W\left(\frac{\ell}{L}\right) W\left(\frac{m}{M}\right)  V\left(\frac{\ell m}{N}\right) .
\]
Here  $W$ is a fixed smooth function such that $\supp W\subset [1/4,2]$ and $W^{(k)} \ll 1$ for $k\geq0$.
Because of $\supp V\subset (1/20,20)$, we only need to consider the case $LM \asymp N$, in which case we can remove the weight function $V$ by a Mellin inversion.
Hence we obtain
\begin{equation}\label{eqn:B(N)<<B(LM)}
  \mathcal{B}(N) \ll T^\varepsilon \sup_{\substack{L\ll N, \ M\ll N \\ LM\asymp N}} |\mathcal{B}(L,M)|,
\end{equation}
where
\[
  \mathcal B(L,M) := \sum_{\ell\geq 1}  \sum_{m\geq1} A(1,m) e\left(- 4(\ell mX)^{1/4}\right) U\left(\frac{\ell}{L}\right) W\left(\frac{m}{M}\right) .
\]
Here $U, W$ are $T^\varepsilon$-inert functions with compact  support.
We have the following estimates.
\begin{proposition}\label{prop:B}
   Let $X^{5/14+\varepsilon} \leq T \leq X^{5/12-\varepsilon}$ and  $ T^4/X \asymp LM$.  Then we have
  \[
    \mathcal B(L,M) \ll T^{53/15+\varepsilon}X^{-5/6}.
  \]
\end{proposition}

  For  $X^{5/14+\varepsilon} \leq T \leq X^{5/12-\varepsilon}$, by \eqref{eqn:I<<sum}, \eqref{eqn:B(N)<<B(LM)}, and Proposition \ref{prop:B}, we have
  \[
    \mathcal I \ll T^\varepsilon
    \sup_{N\asymp T^{4}/X}
    \frac{T^{1/2}}{ N^{1/2}} T^{53/15+\varepsilon}X^{-5/6}
    + T^{1+\varepsilon}
    \ll T^{61/30+\varepsilon} X^{-1/3},
  \]
  as claimed. This completes the proof of Proposition \ref{prop:I}.

\subsection{Estimates of bilinear forms}

In this subsection, we make use of the bilinear structure to prove Proposition \ref{prop:B}, which is an improvement of \cite[Prop. 21]{huang2021rankin}.

\subsubsection{Large $L$}
  If $L\geq T^{1/2+\eta}$ for some $\eta>0$, then by the Poisson summation formula, we have
  \begin{align*}
    \mathcal B(L,M) &
    =  \sum_{m\geq1} A(1,m) W\left(\frac{m}{M}\right)
    \sum_{\ell\in\mathbb{Z}} \int_{\mathbb{R}} e\left(- 4(y mX)^{1/4}\right) U\left(\frac{y}{L}\right) e(-\ell y) \dd y .
  \end{align*}
  Making a change of variable $y=L\xi$, we get
  \begin{align*}
    \mathcal B(L,M) &
    = L \sum_{m\geq1} A(1,m) W\left(\frac{m}{M}\right)
    \sum_{\ell\in\mathbb{Z}} \int_{\mathbb{R}} e\left(- 4(L mX)^{1/4} \xi^{1/4} - \ell L\xi\right) U\left(\xi\right) \dd \xi.
  \end{align*}
  By repeated integration by parts we know the contribution is negligibly small unless $\ell \asymp T/L$, in which case by the second derivative test (Lemma \ref{lemma:stationary_phase}) we have
  \[
    \int_{\mathbb{R}} e\left(- 4(L mX)^{1/4} \xi^{1/4} - \ell L\xi\right) U\left(\xi\right) \dd \xi
    \ll \frac{1}{\sqrt{T}}.
  \]
  Hence we obtain
  \begin{equation} \label{eqn:largeL}
    \mathcal B(L,M) \ll L
    \sum_{m\asymp M} |A(1,m)|
    \sum_{\ell\asymp T/L} T^{-1/2} + T^{-2023}
    \ll
    T^{1/2} M \ll T^{4-\eta} /X .
  \end{equation}
  Here we have used \eqref{eqn:RS3}.

  \subsubsection{Medium $L$}
  If $L\leq T^{1/2+\eta'}$ with $\eta'\in[0,1/3]$, then
  \[
    M\asymp N/L \gg (T^{4}/X) /T^{1/2+\eta'}  \asymp T^{7/2-\eta'}/X  \geq T^{7/2-\eta'-14/5-\varepsilon}  \geq T^{3/10}
  \]
  provided $T\geq X^{5/14+\varepsilon}$.
  By the Cauchy--Schwarz inequality, we have
  \begin{align*}
    \mathcal B(L,M) ^2 & \leq   \bigg( \sum_{m\geq1} \Big|A(1,m) W\left(\frac{m}{M}\right)\Big| \  \Big| \sum_{\ell\geq 1} e\left(- 4(\ell mX)^{1/4}\right) U\left(\frac{\ell}{L}\right) \Big| \bigg)^2 \\
    & \ll M^{1+\varepsilon}   \sum_{m\geq1}   W\left(\frac{m}{M}\right) \ \Big| \sum_{\ell\geq 1} e\left(- 4(\ell mX)^{1/4}\right) U\left(\frac{\ell}{L}\right) \Big|^2  \\
    & = M^{1+\varepsilon}  \sum_{\ell\geq 1}\sum_{\ell'\geq 1} U\left(\frac{\ell}{L}\right)  U\left(\frac{\ell'}{L}\right)
     \sum_{m\geq1}   W\left(\frac{m}{M}\right)   e\left(- 4(\ell mX)^{1/4}+ 4(\ell' mX)^{1/4}\right) .
  \end{align*}
  By the Poisson summation formula, the innermost $m$-sum is equal to
  \begin{align*}
    \sum_{m\in\mathbb{Z}} \int_{\mathbb{R}}  W\left(\frac{y}{M}\right)   e\left((- 4(\ell X)^{1/4}+ 4(\ell' X)^{1/4}) y^{1/4} - my\right) \dd y.
  \end{align*}
  By making a change of variable $y=M\xi$, we get
  \begin{equation}\label{eqn:xi-integral}
    M \sum_{m\in\mathbb{Z}} \int_{\mathbb{R}}  W\left(\xi\right)   e\left(4 X^{1/4} M^{1/4} \frac{\ell'-\ell}{\ell^{3/4}+\ell^{1/2}\ell'^{1/4}+\ell^{1/4}\ell'^{1/2}+\ell'^{3/4}} \xi^{1/4} - mM \xi\right) \dd \xi.
  \end{equation}

  We fist consider the zero frequency $m=0$. If $\ell\neq\ell'$, then we have
  \begin{equation*}
    X^{1/4} M^{1/4} \frac{\ell'-\ell}{\ell^{3/4}+\ell^{1/2}\ell'^{1/4}+\ell^{1/4}\ell'^{1/2}+\ell'^{3/4}}
    \gg X^{1/4} M^{1/4} \frac{1}{L^{3/4}} \asymp T/L \gg T^{1/2-\eta'} \geq T^{1/6}.
  \end{equation*}
  By repeated integration by parts, we know the contribution is negligibly small.
  If $\ell=\ell'$, then this contribution to $\mathcal B(L,M) ^2$ is $O(LM^{2+\varepsilon})$.
  Hence the contribution from the zero frequency $m=0$ to $\mathcal B(L,M)$ is bounded by
  \begin{align} \label{eqn:middleL0}
    O(L^{1/2} M^{1+\varepsilon})=O( L^{-1/2} T^{4+\varepsilon} X^{-1}).
  \end{align}

  Now we consider the nonzero frequency $m\neq0$.
  The contribution from the nonzero frequency  to $\mathcal B(L,M)$ is bounded by $O(T^{-2023})$ if $M\geq T^{1+\varepsilon}$. Now we assume  $M\leq T^{1+\varepsilon}$.
  If $m\gg T^{1+\varepsilon}/M$, then the contribution is negligibly small by the repeated integration by parts.
  We can insert a dyadic partition of unity for the $m$-sum. So we need to consider the case $m\asymp R$ with $1\ll R\ll T^{1+\varepsilon}/M$.
  Note that we have
  $mM\asymp RM \gg T^{3/10}$. By the stationary phase method, we know that the $\ell$ and $\ell'$ sums can be restricted to
  $ \ell-\ell'\asymp RML/T$,
  in which case the second derivative test shows that the $\xi$-integral in \eqref{eqn:xi-integral} is bounded by $O(1/\sqrt{RM})$.
  Hence the nonzero frequency $m\neq0$ to $\mathcal B(L,M)$ is
  \begin{align} \label{eqn:middleL1}
    & \ll \bigg( M^{2+\varepsilon} \sup_{1\ll R\ll T^{1+\varepsilon}/M} \mathop{\sum_{\ell\asymp L}\sum_{\ell'\asymp L}}_{\ell-\ell'\asymp RML/T} \sum_{ m\asymp R  } \frac{1}{\sqrt{RM}}
    \bigg)^{1/2} + T^{-2023} \nonumber \\
    & \ll L  M^{1/2} T^{1/4+\varepsilon}
    \asymp L^{1/2} T^{9/4+\varepsilon} X^{-1/2}.
  \end{align}
  By  \eqref{eqn:middleL0} and \eqref{eqn:middleL1} we have
  \begin{equation}\label{eqn:mediumL}
    \mathcal B(L,M) \ll L^{-1/2} T^{4+\varepsilon} X^{-1} + L^{1/2} T^{9/4+\varepsilon} X^{-1/2}.
  \end{equation}

  \subsubsection{Small $L$}\label{subsubsec:smallL}
  If $L\leq T^{14/5}/X$, then we have $M\asymp N/L \gg T^{6/5}$.
  Note that $M\leq N \ll T^{4}/X \ll T^{8/5-\varepsilon}$ provided $T\leq X^{5/12-\varepsilon}$.
  By Lemma  \ref{lemma:3} we have
  \begin{align} \label{eqn:smallL}
    \mathcal B(L,M) & \leq  \sum_{ \ell\asymp  L} \Big| \sum_{m\geq1}  A(1,m) W\left(\frac{m}{M}\right)  e\left(- 4(\ell mX)^{1/4}\right) \Big|  \nonumber
     \\
     & \ll L T^{3/10+\varepsilon} M^{3/4}
   \ll L^{1/4}  T^{33/10+\varepsilon} X^{-3/4}.
  \end{align}

  \subsubsection{Proof of Proposition \ref{prop:B}}
  Let $L_0:=T^{14/15}X^{-1/3}$ be the solution of
  \[
    L^{-1/2} T^{4} X^{-1} = L^{1/4}  T^{33/10} X^{-3/4}.
  \]
  Note that $L_0\leq T^{14/5}/X$ if $T\geq T^{5/14}$.
  If $L\leq L_0$, then by \eqref{eqn:smallL} we have
  \[
    \mathcal B(L,M) \ll T^{53/15+\varepsilon}X^{-5/6}.
  \]
  If $L_0\leq L\leq T^{1/2+\eta'}$ with $\eta'\in[0,1/3]$, then by \eqref{eqn:mediumL} we have
  \[
    \mathcal B(L,M) \ll T^{53/15+\varepsilon}X^{-5/6} + T^{5/2+\eta'/2+\varepsilon}X^{-1/2}
    \ll T^{53/15+\varepsilon}X^{-5/6} ,
  \]
  provided $T^{31/10-3\eta'/2}\geq X$. We should take $\eta'$ to the solution of  $(31/10-3\eta'/2)5/14=1$, that is, $\eta'=1/5$.
  If $L\geq T^{1/2+\eta}$, then by \eqref{eqn:largeL} we have
  \[
    \mathcal B(L,M) \ll T^{4-\eta}X^{-1} \ll T^{53/15+\varepsilon}X^{-5/6},
  \]
  provided $T^{7/15-\eta}\leq X^{1/6}$. Take $\eta$ to be the solution of $(7/15-\eta)5/12=1/6$, i.e., $\eta=1/15$.
  Note that  $\eta'>\eta$.
  This completes the proof of Proposition \ref{prop:B}.


\section*{Acknowledgements}
%
The author would like to thank Yongxiao Lin for sharing the preprint \cite{KLM}, which led to this improvement. He also wants to take this opportunity to thank the referees for his paper \cite{huang2021rankin}, which he forgot to write down in the published version.


\end{document}